\documentclass[11pt]{amsart}

\usepackage[T1]{fontenc}
\usepackage{hyperref}     
\usepackage{geometry}         
\usepackage{amsmath,amsfonts,amssymb,amsthm}
\usepackage{mathrsfs}          
\usepackage{dsfont}
\usepackage{tikz}									

\newtheorem{thm}{Theorem}[section]

\newtheorem{propo}[thm]{Proposition}
\newtheorem{lem}[thm]{Lemma}
\newtheorem{cor}[thm]{Corollary}

\theoremstyle{definition}
\newtheorem{de}[thm]{Definition}
\newtheorem{ntc}[thm]{Notation}
                        
\newtheorem{ex}[thm]{Example}

\theoremstyle{remark}
\newtheorem{rmk}[thm]{Remark}

\numberwithin{equation}{section}

\newcommand{\CC}{\mathds{C}}
\newcommand{\KK}{\mathds{K}}
\newcommand{\RR}{\mathds{R}}
\newcommand{\NN}{\mathds{N}}
\newcommand{\ZZ}{\mathds{Z}}

\newcommand{\FF}{\mathds{F}}
\newcommand{\PP}{\mathds{P}}
\newcommand{\TT}{\mathds{T}}
\newcommand{\BB}{\mathds{B}}
\newcommand{\A}{\mathcal{A}}
\renewcommand{\P}{\mathcal{P}}

\newcommand{\V}{\mathcal{V}}

\newcommand{\C}{\mathcal{C}}
\newcommand{\M}{\mathcal{M}}
\newcommand{\N}{\mathcal{N}}

\newcommand{\T}{\mathcal{T}}
\newcommand{\U}{\mathcal{U}}

\renewcommand{\L}{\mathcal{L}}
\newcommand{\p}{\pi_1}
\newcommand{\inv}{^{-1}}

\newcommand{\set}[1]{\left\{ #1 \right\}}
\renewcommand{\epsilon}{\varepsilon}

\renewcommand{\hat}{\widehat}
\renewcommand{\tilde}{\widetilde}

\DeclareMathOperator{\Hom}{Hom}
\DeclareMathOperator{\HH}{H}

\DeclareMathOperator{\corank}{corank}
\DeclareMathOperator{\im}{Im}
\DeclareMathOperator{\coker}{coker}

\DeclareMathOperator{\sgn}{sgn}

\DeclareMathOperator{\DEPTH}{depth}
\DeclareMathOperator{\depth}{\overline{depth}}

\DeclareMathOperator{\supp}{supp}
\DeclareMathOperator{\Tub}{Tub}

\DeclareMathOperator{\Braid}{B}

\newcommand{\ncross}[2]{\draw[thick,cap=round](#2,#1)--(1+#2,#1);}
\newcommand{\ocross}[2]{\draw[thick,cap=round](#2,#1)--(1+#2,1+#1); \draw[thick,color=white,line width=6pt] (0.2+#2,1+#1)--(0.8+#2,#1); \draw[thick,cap=round](#2,1+#1)--(1+#2,#1);}
\newcommand{\ucross}[2]{\draw[thick,cap=round](#2,1+#1)--(1+#2,#1); \draw[thick,color=white,line width=6pt] (0.2+#2,#1)--(0.8+#2,1+#1); \draw[thick,cap=round](#2,#1)--(1+#2,1+#1);}
\newcommand{\rcross}[2]{\draw[thick,cap=round](#2,1+#1)--(1+#2,#1); \draw[thick,cap=round](#2,#1)--(1+#2,1+#1);}
\newcommand{\mcross}[3]{\foreach \k in {1,2,...,#2} {\draw[thick,cap=round](#3,-1+\k+#1)--(1+#3,-\k+#1+#2);} }

\begin{document}

\title{A topological invariant of line arrangements}

\author[E. Artal]{E. Artal Bartolo}
\address{Departamento de Matem{\'a}ticas-IUMA, Universidad de Zaragoza, 50009 Zaragoza SPAIN }
\email{artal@unizar.es}
\author{V. Florens}
\address{LMA, UMR CNRS 5142 Universite de Pau et des Pays de l'Adour
64000 Pau FRANCE}
\email{vincent.florens@univ-pau.fr}
\thanks{First named author is partially supported by
MTM2010-2010-21740-C02-02 and Grupo Consolidado Geometr{\'i}a E15; two last named authors
are partially supported by ANR Project Interlow ANR-09-JCJC-0097-01}
\author{B. Guerville-Ball{\'e}}
\address{LMA, UMR CNRS 5142 Universite de Pau et des Pays de l'Adour
64000 Pau FRANCE}
\email{benoit.guerville@univ-pau.fr}

\subjclass[2010]{32S22, 52C35, 55N25, 57M05}

\begin{abstract}
We define a new topological invariant of line arrangements in the complex projective plane. 
This invariant is a root of unity defined under some combinatorial restrictions for arrangements endowed with 
some special torsion character on the fundamental group of their complements. 
It is derived from the peripheral structure  on the group induced by the inclusion map of 
the boundary of a tubular neigborhood in the exterior of the arrangement. 
By similarity with knot theory, it can be viewed as an analogue of linking numbers. 
This is an orientation-preserving invariant for ordered arrangements. 
We give an explicit method to compute the invariant from the equations of the arrangement, 
by using wiring diagrams introduced by Arvola, that encode the braid monodromy.
Moreover, this invariant is a crucial ingredient to compute the depth of a character 
satisfying some resonant conditions, and complete the existent methods by Libgober and the first author.
Finally, we compute the invariant for extended MacLane arrangements with an additional line 
and observe that it takes different values for the deformation classes.
\end{abstract}
\maketitle

%%%%%%%%%%%%%%%%%%%%%%%%%%%%%%%%%%%%%%%%%%%%%
%%%%%%%%%%%%%%%%%%%%%%%%%%%%%%%%%%%%%%%%%%%%%
\section*{Introduction}\label{Section_Introduction}
%%%%%%%%%%%%%%%%%%%%%%%%%%%%%%%%%%%%%%%%%%%%%
%%%%%%%%%%%%%%%%%%%%%%%%%%%%%%%%%%%%%%%%%%%%%

The influence of the combinatorial data on the topology of a 
projective line arrangement is not at present well understood. 
In the eighties, Orlik and Solomon \cite{OrlSol} showed that the cohomology 
ring of the complement of an arrangement is determined by the description of the 
incidence relations between the multiple points. 
This is not true for the deformation classes, as it was shown for MacLane combinatorics \cite{Mac,matroids}.
Rybnikov~\cite{Ryb,ry:11,ArtCarCogMar_Invariants} 
constructed a pair of complex line arrangements with the same combinatorics 
but whose fundamental groups are not isomorphic. 
This illustrates that the combinatorics of an arrangement $\A$ does 
not determine in general the homeomorphism type of the pair $(\CC \PP^2, \A)$. 
Other examples were exhibited by the first author, Carmona, Cogolludo and Marco 
by using the braid monodromy \cite{ArtCarCogMar_Topology}.

One of the strongest invariant of the topology of an arrangement is
the fundamental group of its complement, which can be computed using Zariski-van~Kampen method,
see also the specific approach by Arvola~\cite{Arv}. Even when the fundamental
group can be computed, it is very difficult to handle directly with.
This is why a number of invariants derived from the group have emerged, 
as the Alexander invariant and the characteristic varieties. These invariants can
be computed from the fundamental group, but the task can be endless for present computers
for most arrangements. A weaker invariant is the one-variable Alexander polynomial
(also known as characteristic polynomial of the monodromy of the Milnor fiber)
which can be computed directly from the fundamental group or by a general method
available for any projective curve. Though there are some partial results, 
it is still unknown whether characteristic varieties or Alexander polynomial are combinatorially determined.

Using a method by Ligboger~\cite{Lib_Characteristic} it is possible to compute \emph{most} irreducible
components of characteristic varieties (only some isolated points may fail to be found). A method
to compute these extra components can be found in~\cite{Art_Singularities}.

The \emph{boundary manifold} $B_{\A}$, defined as the common boundary of a regular 
neighbourhood of $\A$ and its exterior $E_\A$, is a graph $3$-manifold whose structure 
is determined by the combinatorics of $\A$. In the present paper, we construct 
a new topological invariant of line arrangements derived from the 
\emph{peripheral structure} $\pi_1(B_\A) \rightarrow \pi_1(E_{\A})$, induced by the inclusion map,
and previously studied in~\cite{FloGueMar}.  
This invariant is a root of unity defined for triples composed by an arrangement, a torsion character in $\CC^*$
and a cycle in the incidence graph of the arrangement; there are combinatorial restrictions for
the availability of this invariant, in particular the cycle must be non trivial and satisfy some resonant conditions. 
As it was remarked in \cite{Gue_Thesis}, it is in fact extracted from the homological reduction 
$i_*: H_1(B_\A) \rightarrow H_1(E_\A)$, which is a more tractable object, and corresponds to the value of the 
character on certain homology classes of the boundary manifold, viewed in $E_\A$ throught the inclusion map. 
This construction has similarities with knot theory, and the invariant is a sort of analogue of linking numbers.
As we state in the first main result,
Theorem~\ref{Theorem_Invariant}, it is an \emph{ordered-oriented} topological invariant.

As a second main contribution of this work, we give an explicit method to compute the invariant, 
in terms of braid monodromy. 
We use \emph{braided wiring diagrams} introduced by Arvola~\cite{Arv} (see also Suciu-Cohen~\cite{CohSuc_Monodromy}) 
to encode the braid monodromy relative to a generic projection of the arrangement. 
Note that this invariant is most probably of algebraic nature, even though our computations are topological.

It appears that this invariant is crucial for the computation of the
\emph{quasi-projective} depth of a (torsion) character in~\cite{Art_Singularities}.
The knowledge of depths of all characters
is equivalent to the knowledge of characteristic varieties;
the depth can be decomposed into a \emph{projective term} and a \emph{quasi-projective term}, 
vanishing for characters that ramify along all the lines. 
An algorithm to compute the projective part was given by Libgober
(see \cite{Lib_Characteristic} and also \cite{Art_Singularities} for details). 
An explicit way to compute the quasi-projective depth of resonant (torsion) characters
is given in~\cite{Art_Singularities}, and it happens that the invariant in this paper
is crucial for that method.
Hence, our invariant may help to find examples of combinatorially equivalent arrangements with different structure
for their characteristic varieties, though we have failed till now in finding such arrangements.
But as we show in this paper, this invariant is interesting in its own.

We compute the invariant for MacLane arrangements with an additional line, 
and observe that it takes different values for the two deformation classes. 
This shows that it provides information on their topologies, not contained in the combinatorics.
In particular, there is no ordered-oriented homeomorphism between both realizations; note that
this fact is a consequence of the same result for MacLane arrangements (as shown by Rybnikov).

In a forthcoming paper, the third author present new examples of \emph{Zariski pairs} of arrangements with $12$ lines
(of different nature of previous examples), where this invariant is the unique known one to show that they have not the same topological type.

In \S\ref{Section_InnerCyclicTriplet} we introduce the combinatorial and topological objects
to be used in the paper, namely the concept of inner-cyclic triplets. In \S\ref{Section_Invariant},
we define the invariant of the realization of an inner-cyclic triplet and we prove its topological invariance.
In \S\ref{Section_CharacteristicVarieties}, the relationship of the invariant with
characteristic varieties is described. The topological computation of the invariant via
wiring diagrams and the results of~\cite{FloGueMar} is in \S\ref{Section_Computability}.
Finally, \S\ref{Section_Examples} computes the invariant for the two ordered
realizations of the extended MacLane combinatorics.

%%%%%%%%%%%%%%%%%%%%%%%%%%%%%%%%%%%%%%%%%%%%%
%%%%%%%%%%%%%%%%%%%%%%%%%%%%%%%%%%%%%%%%%%%%%
\section{Inner cyclic triplets}\label{Section_InnerCyclicTriplet}
%%%%%%%%%%%%%%%%%%%%%%%%%%%%%%%%%%%%%%%%%%%%%
%%%%%%%%%%%%%%%%%%%%%%%%%%%%%%%%%%%%%%%%%%%%%

\subsection{Combinatorics}
\mbox{}

In what follows, a brief reminder on line combinatorics is given, see~\cite{ArtCarCogMar_Topology} 
for details. We also introduce the notion of \emph{inner cyclic}.

\begin{de}
	A \emph{combinatorial type}, or simply a (\emph{line}) \emph{combinatorics}, 
is a couple $\C=(\L,\P)$, where $\L$ is a finite set and $\P \subset \P(\L)$, satisfying that:
	\begin{itemize}
		\item For all $p \in \P,\ \sharp p \geq 2$;
		\item For any $\ell_1,\ell_2 \in \L,\ \ell_1\neq\ell_2,
\ \exists ! p\in \P$ such that $\ell_1,\ell_2\in p$.
	\end{itemize}
An \emph{ordered combinatorics} $\C$ is a combinatorics where $\L$ is an ordered set.
\end{de}

This notion encodes the intersection pattern of a collection of lines
in a projective planes, see~\S\ref{subsec:real}, where the relation~$\in$ corresponds
to the dual plane. There are several ways to encode
a line combinatorics.

\begin{de}\label{Definition_IncidenceGraph}
	 The \emph{incidence graph} $\Gamma_{\C}$ of a line combinatorics $\C=(\L,\P)$ 
is a non-oriented bipartite graph where the set of vertices $V(\C)$ decomposes as $V_P(\C)\amalg V_L(\C)$, with:
	\begin{equation*}
		V_P(\C)=\{v_p\mid p\in \P\},\quad \text{and}\quad
		V_L(\C)=\{v_\ell\mid \ell\in \L\}.
	\end{equation*}
	The vertices of $V_P(\C)$ are called \emph{point-vertices} and those of $V_L(\C)$ 
are called \emph{line-vertices}.
	An edge of $\Gamma_{\C}$ joins $v_\ell$ to $v_p$ if and only if $\ell\in p$. 
Such an edge is denoted by $e(\ell,p)$.
\end{de}

For a line arrangement in the projective plane the incidence graph is the dual graph of the divisor
obtained by the preimage of the line arrangement in the blowing-up of the projective plane along
the set of multiples points of the arrangement, see~\S\ref{subsec:real}.

\begin{de}
A \emph{character} on a line combinatorics $(\L,\P)$ 
is a map $\xi:\L\rightarrow \CC^*$ such that 
\begin{equation}\label{eq:prod_total}
\prod_{\ell\in\L} \xi(\ell)=1.
\end{equation}
A \emph{torsion character} on a line combinatorics $(\L,\P)$ 
is a character $\xi$ where for all $\ell\in\L$, $\xi(\ell)$ is a root of unity.
\end{de}

Namely, we are associating a non-zero complex number to each element of $\L$, such that the product
of all of them equals~$1$. These characters have a cohomological meaning for line arrangements in the complex
projective plane. 

\begin{de}\label{def:car_point}
Let $\xi$ be a character on a line combinatorics $\C=(\L,\P)$.
For each $p\in\P$, we define  $\xi(p):=\prod_{\ell\in p} \xi(\ell)$.
\end{de}

A \emph{cycle} of $\Gamma_\C$ is an element of $H_1(\Gamma_\C)$. Finally, we introduce the main object of this work.

\begin{de}\label{Definition_InnerCyclic}
	An \emph{inner cyclic triplet} $(\C,\xi,\gamma)$ is a line combinatorics $\C=(\L,\P)$, 
a torsion character $\xi$ on $\L$ and a cycle $\gamma$ of $\Gamma_\C$ such that:
	\begin{enumerate}
		\item\label{Definition_InnerCyclic:1} for all line-vertex $v_\ell$ of $\gamma$, $\xi(\ell)=1$,
		\item\label{Definition_InnerCyclic:2} for all point-vertex $v_p$ of $\gamma$, and for all $\ell\in p$, $\xi(\ell)=1$,
		\item\label{Definition_InnerCyclic:3} for all $p\in\ell$, with $v_\ell\in\gamma$, $\xi(p)=1$.
	\end{enumerate}
\end{de}

The above conditions can be understood in a shorter way: all the vertices of $\gamma$ and all their
neighbors come from elements $m\in V(\C)$ such that $\chi(m)=1$.

%%%%%%%%%%%%%%%%%%%%%%%%%%%%%%%%%%%%%%%%%%%%%
\subsection{Realization}\label{subsec:real}\mbox{}
%%%%%%%%%%%%%%%%%%%%%%%%%%%%%%%%%%%%%%%%%%%%%

Let $\A$ be a line arrangement in $\CC\PP^2\equiv\PP^2$ and let $\P_\A$ be the set of multiple points
of $\A$; then $\C_\A:=(\A,\P_\A)$ is the \emph{combinatorics} of $\A$. Given a combinatorics
$\C=(\L,\P)$, a \emph{complex realization} of $\C$
is a line arrangement $\A$ in $\PP^2$ such that its combinatorics agrees with $\C$.
An \emph{ordered realization} of an ordered combinatorics is defined accordingly.
The existence of realizations of a combinatorics depends on the field $\KK$. 

Given a line arrangement~$\A$ in $\PP^2$, or more generally a set of irreducible curves
$\A$ in a projective surface~$X$, we will denote $\bigcup\A$ the union of those curves.

\begin{rmk}\label{rmk:realization}
Let $\A$ be a line arrangement in $\PP^2$ with combinatorics $\C=(\L,\P)$ (note that $\L=\A$).
For $\ell\in\L$ and $p\in\P$ the notion $\ell\in p$ can be understood in the dual plane $\check{\PP}^2$.
For convenience we also use the notation $p\in \ell$. Let $\pi:\hat{\PP}^2\to\PP^2$ be the composition
of the blow-ups of the points in $\P$; then $\tilde{\A}:=\pi^{-1}(\A)$ defines a normal crossing
divisor in $\hat{\PP}^2$ whose dual graph is exactly $\Gamma_\C$. Note also that 
$M_\A:=\PP^2\setminus\bigcup\A=\hat{\PP}^2\setminus\bigcup\tilde{\A}$. Let us identify
each $v_\ell\in V_L(\C)$ 
with a meridian of $\ell$ as an element in $H_1(M_\A;\ZZ)$. Note that
$V_L(\C)$ generates $H_1(M_\A;\ZZ)$ and the only relation satisfied by them is~\eqref{eq:prod_total}.
Moreover, if $v_p\in V_P(\C)$ is identified with a meridian of $\pi^{-1}(P)$ as element
of $H_1(M_\A;\ZZ)$, then, the equality
\begin{equation}\label{eq:mer_pt}
v_p=\prod_{p\in\ell} v_\ell=\prod_{\ell\in p} v_\ell
\end{equation}
holds. Via this identification the space of characters of $\C$ coincides with
$$
H^1(M_\A;\CC^*)=\Hom(H_1(M_\A;\ZZ),\CC^*)\cong(\CC^*)^{\sharp\A-1}.
$$
Equation \eqref{eq:mer_pt} agrees with Definition~\ref{def:car_point}.
\end{rmk}

The space $M_\A$ is not compact and this may cause some trouble. To avoid this complication,
let $\Tub(\A)$ be a compact regular neighbourhood of $\A$ in $\PP^2$, and 
let $E_\A:=\overline{\PP^2\setminus \Tub(\A)}$ be the \emph{exterior} of the arrangement.
This is an oriented $4$-manifold with boundary such that the inclusion
$E_\A\hookrightarrow M_\A$ is a homotopy equivalence
and in particular $H_1(E_\A;\ZZ)\equiv H_1(M_\A;\ZZ)$.
As we have seen, if $\sharp\A=n+1$, then $H_1(E_\A;\ZZ)\cong\ZZ^n$ and
it is freely generated by the meridians of any subset of $n$ lines in $\A$.

\begin{ntc}
We use the notation $\Gamma_\A$ for the incidence graph of $\Gamma_\C$.
\end{ntc}

\begin{de}
	Let $(\C,\xi,\gamma)$ be a triplet, where $\C$ is an ordered combinatorics,
$\xi$ a character and $\gamma$ a circular cycle of $\Gamma_\C$. A \emph{realization} of $(\C,\xi,\gamma)$ is
a triplet $(\A,\xi_\A,\gamma_\A)$:
	\begin{itemize}
		\item An ordered realization $\A$ of $\C$;
		\item A character $\xi_\A:\HH_1(E_\A;\ZZ)\rightarrow\CC^*$ such that $\xi_\A(v_\ell)=\xi(\ell)$
under the identification of Remark~\ref{rmk:realization}.
		\item A cycle $\gamma_\A$ in $\Gamma_\A$ which coincides with $\gamma$ via the natural
identification $\Gamma_\A\equiv\Gamma_\C$.
	\end{itemize}
Due to this natural identifications we usually drop the subindex $\A$.
If $(\C,\xi,\gamma)$ is an inner cyclic triplet, then the realization $(\A,\xi,\gamma)$ is \emph{inner cyclic}.
\end{de}

%%%%%%%%%%%%%%%%%%%%%%%%%%%%%%%%%%%%%%%%%%%%%
\subsection{Regular neighborhoods}\mbox{} 
%%%%%%%%%%%%%%%%%%%%%%%%%%%%%%%%%%%%%%%%%%%%%

Let $\A$ be a line arrangement with combinatorics $\C$. Let us describe how to construct
a compact regular neighbourhood $\Tub(\A)$. There are several ways
to define it, see \cite{CohSuc_Boundary,Dur}, and they produce isotopic results.

For each $\ell\in\A$ we consider
a tubular neighborhood $\Tub(\ell)$ of $\ell\subset\PP^2$ and for each $p\in\P$ we consider
a closed $4$-ball $\BB_p$ centered at $p$.

\begin{de}
We say that the set $\{\Tub(\ell)\mid \ell\in\A\}\cup\{\BB_p\mid p\in\P\}$ is a
\emph{compatible system of neighborhoods in $\A$} if $\forall \ell_1\neq\ell_2$ we have
$$
\Tub(\ell_1)\cap\Tub(\ell_2)=\mathring{\BB}_p,\quad p=\ell_1\cap\ell_2
$$
and the balls are pairwise disjoint.
The union of these neighborhoods is a \emph{regular neighborhood of} $\A$.
Given such a system, for each $\ell$ we define the \emph{holed neighborhood} 
$$
\mathcal{N}(\ell):=\overline{\Tub(\ell)\setminus\bigcup_{p\in\ell} \BB_p}.
$$
\end{de}

For each $\ell\in\A$ we denote by
$$
\check{\ell}:=\overline{\ell\setminus\bigcup_{p\in\ell} \BB_p}.
$$
This is a punctured sphere (which as many punctures as the number of multiple points in $\ell$)
and the space $\mathcal{N}(\ell)$ is (non-naturally) homeomorphic to $\check{\ell}\times D^2$,
where $D^2$ is a closed disk in $\CC$.

Let us consider now $\pi:\hat{\PP}^2\to\PP^2$. For each $p\in\P$, $\pi^{-1}(\BB_p)$ is a tubular
neighborhood of the rational curve $E_p:=\pi^{-1}(p)$; this is a locally trivial $D^2$-bundle, but not trivial. 
Note that $\pi^{-1}(\mathcal{N}(\ell))$ is naturally isomorphic
to $\mathcal{N}(\ell)$ and that $\Tub(\tilde{\A}):=\pi^{-1}(\Tub(\A))$ is a regular neighborhood of
$\tilde{\A}$ obtained by plumbing the tubular neighborhood of its
irreducible components.

If $p\in\ell$ then $V_{\ell,p}:=\mathcal{N}(\ell)\cap\BB_p$ is a tubular neighborhood of the trivial knot
$\ell\cap\partial\BB_p\subset\partial\BB_p$, i.e. a solid torus.

The manifolds $E_\A$ and $\Tub(\A)$ share their boundaries $B_\A=E_\A\cap\Tub(\A)$ and note also that
$B_\A$ can be identified with the boundary of $\Tub(\tilde{\A})$. This manifold
is a graph manifold obtained by gluing the following pieces:
\begin{equation*}
B_\ell:=\overline{\partial\N(\ell)\setminus\bigcup_{p\in\ell} V_{\ell,p}},\qquad
B_p:=\overline{\partial\BB_p\setminus\bigcup_{\ell\in p} V_{\ell,p}}.
\end{equation*}
Note also that $B_\A\hookrightarrow\Tub(\A)\setminus\bigcup\A$ is a homotopy equivalence.

%%%%%%%%%%%%%%%%%%%%%%%%%%%%%%%%%%%%%%%%%%%%%
\subsection{Nearby cycles}\mbox{} 
%%%%%%%%%%%%%%%%%%%%%%%%%%%%%%%%%%%%%%%%%%%%%

Let $\A$ be an arrangement, and $\gamma$ be a circular cycle of $\Gamma_\A$.  
The \emph{support} of $\gamma$ is defined as:
\begin{equation*}
	\supp(\gamma)=\set{\ell\in\A\ \mid\ v_\ell\in\gamma}=\set{\ell_1,\dots,\ell_r}.
\end{equation*}
with cyclic order $\ell_1<\dots<\ell_r$ and $\ell_{r+1}:=\ell_1$. Let
$p_j:=\ell_j\cap\ell_{j+1}$, $j=1,\dots,r$.

\begin{de}\label{def:embedding}
An \emph{embedding} of $\gamma$ in $\A$ is a simple closed loop $r(\gamma)\subset\bigcup\A$
defined as follows. Take a point $q_j\in\check{\ell}_j$ ($q_{r+1}:=q_1$), $j=1,\dots,r$.
We denote by $p_j^\alpha$ a point in $\ell_j\cap\BB_{p_j}$
and by $p_j^\omega$ a point in $\ell_{j+1}\cap\BB_{p_j}$. Let $\rho_j^\alpha$ be a radius
in $\ell_j$ from $p_j$ to $p_j^\alpha$ and $\rho_j^\omega$ be a radius
in $\ell_{j+1}$ from $p_j$ to $p_j^\omega$.
Pick up arbitrary simple paths
$\alpha_j$ from $q_j$ to $p_j^\alpha$ in $\check{\ell}_j$, $j=1,\dots,r$, and
$\omega_j$ from $p_{j-1}^\omega$ to $q_j$, $j=2,\dots,r+1$.
Then:
$$
r(\gamma):=\alpha_1\cdot(\rho_1^\alpha)^{-1}\cdot\rho_1^\omega\cdot\omega_2\cdot\alpha_2\cdot\ldots
\cdot(\rho_r^\alpha)^{-1}\cdot\rho_r^\omega\cdot\omega_{r+1}.
$$
\end{de}

\begin{de}\label{Definition_Nearby}
A \emph{nearby cycle} $\tilde{\gamma}$ associated with $\gamma$ is a smooth path in $B_\A$,
homologous to an embedding $r(\gamma)$ of $\gamma$ in $\Tub(\A)$, lying in
$$
\left(\bigcup_{j=1}^r\N(\ell_j)\cup\bigcup_{j=1}^r\BB_{p_j}\right)\setminus\bigcup\A.
$$
\end{de}

\begin{rmk}
Note that if $p\neq p_j$ a nearby cycle $\tilde{\gamma}$ can intersect $\BB_p$ only at some $V_{\ell_k,p}$
for some $k$; there are always nearby cycles which do not intersect $\BB_p$ for  $p\neq p_j$.
\end{rmk}

%%%%%%%%%%%%%%%%%%%%%%%%%%%%%%%%%%%%%%%%%%%%%
%%%%%%%%%%%%%%%%%%%%%%%%%%%%%%%%%%%%%%%%%%%%%
\section{Invariant}\label{Section_Invariant}
%%%%%%%%%%%%%%%%%%%%%%%%%%%%%%%%%%%%%%%%%%%%%
%%%%%%%%%%%%%%%%%%%%%%%%%%%%%%%%%%%%%%%%%%%%%

Let $(\C,\xi,\gamma)$ be a inner-cyclic triplet; and suppose that $(\A,\xi, \gamma)$ is a realization. 
Denote by $i$ the inclusion map of the boundary manifold in the exterior, i.e.,
$i : B_\A \hookrightarrow E_\A$.  We consider the following composition map:
\begin{equation*}
	\chi_{(\A,\xi)}:
		 \HH_1(B_\A)  \overset{i_*}{\longrightarrow}  \HH_1(E_\A)  \overset{\xi}{\longrightarrow}  \CC^*
\end{equation*}

Let $\tilde{\gamma}$ be a nearby cycle associated with $\gamma$. 
As we are going to prove in Lemma~\ref{Lemma_WellDefine}, we can define 
$$
\mathcal{I}(\A,\xi, \gamma) := \chi_{(\A,\xi)} (\widetilde{\gamma}).
$$
Our main goal is to prove that this is a topological invariant, as we state here.

\begin{thm}[Main result]\label{Theorem_Invariant}
	If $(\A,\xi,\gamma)$  and $(\A',\xi,\gamma)$ are two inner-cyclic realizations of $(\C,\xi,\gamma)$ with the same (oriented and ordered)
	 topological type, then 
	 \begin{equation*}
		\mathcal{I}(\A,\xi, \gamma) = \mathcal{I}(\A',\xi,\gamma).
	\end{equation*}
\end{thm}

The aim of this section is to prove Theorem \ref{Theorem_Invariant}. 
The first step is to prove the following Lemma which shows that for a given inner-cyclic realization $(\A,\xi,\gamma)$, 
the image of  a nearby cycle $\tilde{\gamma}$ by $\chi_{(\A,\xi)}$ depends only in $\gamma$.

\begin{lem}\label{Lemma_WellDefine}
	Let $(\A,\xi,\gamma)$ be an inner-cyclic realization. If $\tilde{\gamma}$ and $\tilde{\gamma}'$ are nearby cycles associated with $\gamma$, then:
	\begin{equation*}
		\chi_{(\A,\xi)} (\tilde{\gamma})=\chi_{(\A,\xi)} (\tilde{\gamma}').
	\end{equation*}
\end{lem}

\begin{proof}
	The choice of a nearby cycle  in $B_\A$ associated with $\gamma$ depends first on 
the embedding $r(\gamma)$ in $\A$, see Definition~\ref{def:embedding} and use the notation inside.
There is some freedom in the choice of $r(\gamma)$; in $r(\gamma)\cap\ell_j$ we  can
add some meridians around the multiple points of $\A$ in $\ell_j$;
recall that for any such multiple point~$p$ we have $\xi(p)=1$
by Definition~\ref{Definition_InnerCyclic}\eqref{Definition_InnerCyclic:3}.

Once $r(\gamma)$ has been chosen, $\tilde{\gamma}$ lies in $B_\A$, and is homologous to $r(\gamma)$ in $\Tub(\A)$.
Two different representants will differ by the meridians $v_{\ell_j}$ and by cycles
in $\BB_{p_j}\setminus\A$; these ones are linear combinations of meridians
$v_\ell$ for $\ell\in p_j$.

Summarizing the difference of two nearby cycles is a linear combination of
meridians $v_\ell$, where either $\ell\in\supp(\gamma)$ or $\ell\in p_j$,
and meridians $v_p$, $p\in\ell_j$. By Definition~\ref{Definition_InnerCyclic}
all these meridians are in $\ker\xi$ and the result follows.
\end{proof}

We have proved with this Lemma that $\mathcal{I}(\A,\xi, \gamma)$ is well-defined
for a particular $\A$. The following result studies the behavior of this 
invariant under orientation and order preserving homeomorphisms. We prove that
there are nearby cycles of the first realization which are sent to nearby cycles
of the second realization; note that we do not need the stronger result that
would say that the image of \emph{any} nearby cycle is a nearby cycle.

\begin{lem}\label{Lemma_Homotopy}
Let $(\C,\xi,\gamma)$ be an ordered inner-cyclic triplet. 
Suppose that $(\A,\xi, \gamma)$ and $(\A',\xi,\gamma)$ are 
two ordered realizations of  $(\C,\xi,\gamma)$ such that there exists a homeomorphism
$\phi:(\PP^2,\A)\rightarrow (\PP^2,\A')$  preserving orders and orientations of the lines.
Then there is a nearby cycle $\tilde{\gamma}$ associated to $\gamma$ in $B_\A$, 
such that $\phi(\tilde{\gamma})$ is isotopic to a nearby cycle 
associated with $\gamma$ in $B_{\A'}$.
\end{lem}

\begin{proof}
If $\ell$ is a line in $\A$, $\ell'$ is the corresponding line in $\A'$;
we use the same convention for multiple points. 

We start fixing a  regular neighborhood $\Tub(\A')$ obtained as a union of $\BB_{p'}$
and $\N(\ell')$. Since $\phi$ is a homeomorphism we can construct 
a regular neighborhood $\Tub(\A)$ obtained as a union of $\BB_{p}$
and $\N(\ell)$, such that $\phi(\BB_{p})\subset\BB_{p'}$ and $\phi(\N(\ell))\subset\N(\ell')$.

Let us choose an embedding $r(\gamma)$ in $\A$ such that
$r(\gamma)$ is disjoint from $\BB_p$ for $p\neq p_1,\dots,p_r$
which is always possible. Moreover we can construct a continuous
family of nearby cycles $\tilde{\gamma}_t$, $t\in(0,1]$ (converging to $r(\gamma)$ when $t\to 0$)
and such that this family is still disjoint with $\BB_p$ for $p\neq p_1,\dots,p_r$.
Then, $\phi(\tilde{\gamma}_1)$ is a nearby cycle associated with $\gamma$ 
which can be deformed to a nearby cycle in $B_{\A'}$.
\end{proof}

\begin{rmk}
The nearby cycles used in Section~\ref{Section_Computability} are different of the one used in this proof. 
Those cycles may intersect the boundary of $\partial\mathcal{N}$, and this is important
for the computation of the invariant, but it does not affect to our goal. 
\end{rmk}

\begin{proof}[Proof of Theorem{\rm~\ref{Theorem_Invariant}}]
 Suppose that there exists $\phi:(\PP^2,\A)\rightarrow (\PP^2,\A')$  preserving orders and orientations of the lines.
 Hence $\phi$ induces an isomorphism $\phi_* : H_1(\PP^2\setminus\A) \rightarrow H_1(\PP^2\setminus\A') $
  preserving $\xi$. 
Let $\tilde{\gamma}$ be an arbitrary nearby cycle associated with $\gamma$. By Lemma~\ref{Lemma_WellDefine},
  $\mathcal{I}(\A,\xi, \gamma) = \chi_{(\A,\xi)} (\widetilde{\gamma})$ depends only on $\gamma$.
	We can then assume that $\tilde{\gamma}$ is one of the particular nearby cycle considered in Lemma~\ref{Lemma_Homotopy}. But this lemma implies that
	$\phi(\tilde{\gamma})$ is isotopic to a nearby cycle $\tilde{\gamma}'$
associated to $\gamma$ in $B_{\A'}$. It follows that $\chi_{(\A,\xi)} (\widetilde{\gamma}')= \chi_{(\A,\xi)} (\widetilde{\gamma})$, and we conclude using Lemma~\ref{Lemma_WellDefine}.
\end{proof}

\begin{propo}\label{Conjugate_Images}
Let $(\A,\xi,\gamma)$ and $(\overline{\A},\xi,\gamma)$ be two inner cyclic complex conjugated 
realization arrangements. We have then:
	\begin{equation*}
		\mathcal{I}(\A,\xi,\gamma)^{-1}=\mathcal{I}(\overline{\A},\xi,\gamma).
	\end{equation*}
\end{propo}

\begin{proof}
The complex conjugation $(\PP^2,\A)\to(\PP^2,\overline{\A})$ 
is an ordered homeomorphism which preserves the orientation of $\PP^2$ but exchanges
the orientation of the line. For this reason the action of the homeomorphism sends
$\xi$ to $\xi^{-1}$ (since the conjugation sends a meridian $v_\ell$ to $(v_{\bar{\ell}}^{-1}$, with
complex notation). The result follows from the fact that a nearby cycle
for $(\overline{\A},\xi,\gamma)$ it is also a nearby cycle for $(\overline{\A},\xi^{-1},\gamma)$.
\end{proof}

\begin{cor}
	If $\A$ is a complexified real arrangement, 
that is the complexification of a line arrangement defined on $\RR\PP^2$, then:
	\begin{equation*}
		\mathcal{I}(\A,\xi,\gamma)\in\set{-1,1}.
	\end{equation*}
The same result occurs if there is a continuous path of combinatorially equivalent line arrangements connecting
$\A$ and $\overline{\A}$.
\end{cor}

\begin{figure}[hb!]
	\centering
\begin{tikzpicture}
	\begin{scope}[scale=1]
		\draw (-2,1.25)  -- (2,1.25);
		\draw (-2,-1.25) -- (2,-1.25);		
		\draw (-1.25,-2) -- (-1.25,2);
		\draw (1.25,-2) -- (1.25,2);
		\draw (-2,-2) -- (2,2);
		\draw (-2,2) -- (2,-2);
		\node at (-3,0) {$L_0=\infty$};
		\node at (1.25,2.5) {$L_1$};
		\node at (-1.25,2.5) {$L_2$};
		\node at (-2.5,2.5) {$L_3$};
		\node at (-2.5,1.25) {$L_4$};
		\node at (-2.5,-1.25) {$L_5$};
		\node at (-2.5,-2.5) {$L_6$};
	\end{scope}
\end{tikzpicture}
		\caption{The Ceva-7 arrangement  $\mathscr{C}_7$ \label{Ceva7} }
\end{figure}
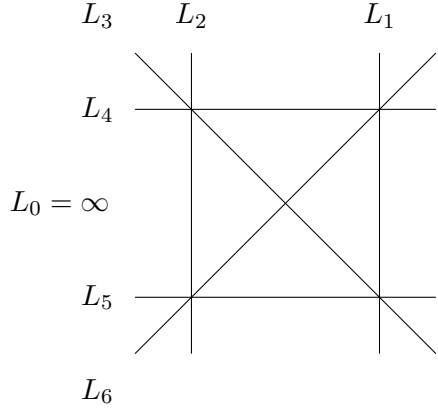

\begin{ex}\label{Example_Invariant}
	Let $\C=(\L,\P)$ be the combinatorics defined by $\L=\set{L_0,\cdots,L_6}$ and 
	\begin{multline*}
		\P=\left\{\set{L_0,L_1,L_2},\set{L_0,L_3},\set{L_0,L_4,L_5},\set{L_0,L_6},\set{L_1,L_3,L_5},\right. \\
						\left.\set{L_1,L_4,L_6},\set{L_2,L_3,L_4},\set{L_2,L_5,L_6},\set{L_3,L_6}\right\}.
	\end{multline*}
The following triplet $(\C,\xi,\gamma)$ is inner-cyclic:
\begin{itemize}	
	\item $\xi$ is the character on $\C=(\L,\P)$ defined by:
	\begin{equation*}
		(L_0,L_1,\cdots,L_6) \longmapsto (1,-1,-1,1,-1,-1,1).
	\end{equation*}
	\item $\gamma$ is the cycle of $\Gamma_\C$ defined by $v_{L_0}\rightarrow v_{P_{0,3}}\rightarrow v_{L_3}\rightarrow v_{P_{3,6}}\rightarrow v_{L_6}\rightarrow v_{P_{0,6}}\rightarrow v_{L_0}$.
 \end{itemize}

	The arrangement $\mathscr{C}_7$ pictured in Figure~\ref{Ceva7} is --up to projective transformation-- 
the only one realization of the combinatorics $\C$. 
There is a nearby cycle $\tilde{\gamma}$ associated with the cycle $\gamma$ 
such that its image by the map $i_*:\HH_1(B_{\mathscr{C}_7})\rightarrow\HH_1(E_{\mathscr{C}_7})$ is $-v_5$
 (an algorithm to compute is given in Section~\ref{Section_Computability}). Then we have:
	\begin{equation*}
		\mathcal{I}(\mathscr{C}_7,\xi,\gamma),\xi)=\xi(-v_5)=-1.
	\end{equation*}
\end{ex}

%%%%%%%%%%%%%%%%%%%%%%%%%%%%%%%%%%%%%%%%%%%%%
%%%%%%%%%%%%%%%%%%%%%%%%%%%%%%%%%%%%%%%%%%%%%
\section{Characteristic varieties}\label{Section_CharacteristicVarieties}
%%%%%%%%%%%%%%%%%%%%%%%%%%%%%%%%%%%%%%%%%%%%%
%%%%%%%%%%%%%%%%%%%%%%%%%%%%%%%%%%%%%%%%%%%%%

The characteristic varieties of an arrangement can be defined as the jumping loci 
of twisted cohomology of rank one local system. They only depend on its fundamental group.
For $\xi\in H^1(M_\A;\CC^*)$, let $\L_\xi$ be local system of coefficients defined
by $\xi$.

\begin{de}
	The \emph{characteristic varieties} of an arrangement $\A$ are:
	\begin{equation*}
		\V_k(\A)=\set{\ \xi\in\TT(\A)\ \mid\ \dim_\CC \left( \HH^1(M_\A;\L_\xi) \right) \geq k\ }.
	\end{equation*}
\end{de}

\begin{de}
	The \emph{depth} of a character $\xi\in\TT(\A)$ is:
	\begin{equation*}
		\DEPTH(\xi)=\max\set{k\in\NN\ \mid\ \xi\in\V_k}=\dim_\CC \HH^1(M_\A;\L_\xi).
	\end{equation*}
\end{de}

Hence for the study of characteristic varieties, we need to be able to compute the twisted cohomology
spaces of $M_\A$.  

%%%%%%%%%%%%%%%%%%%%%%%%%%%%%%%%%%%%%%%%%%%%%
\subsection{Geometric interpretation of the notion of inner cyclic}\mbox{}
%%%%%%%%%%%%%%%%%%%%%%%%%%%%%%%%%%%%%%%%%%%%%

Let us recall that $\pi:\hat{\PP}^2\to\PP^2$ is the blow-up of $\PP^2$ over the points of $\P$; the main goal
of the construction of $\hat{\PP}^2$ is to obtain $\tilde{\A}:=\pi^{-1}(\A)$ as a normal
crossing divisor (in fact, we need only to blow up the points of multiplicity at least three, but it
is harmless to do extra blow-ups).
By construction, we have that $\hat{\PP}^2\setminus\bigcup\hat{\A}\equiv E_\A$, 
then a character $\xi$ on $\p(E_\A)$ can be view as a character on $\p(\hat{\PP}^2\setminus\bigcup\hat{\A})$ 
(also noted $\xi$). Let $\hat{\Gamma}_\A$ be the dual graph of $\hat{\A}$. 

\begin{de}
A component $H\in\hat{\A}$ is \emph{unramified} for 
the character $\xi$ if $\xi(v_H)=1$. 
It is \emph{inner unramified} for $\xi$ if it is unramified and 
all its neighbours (in $\hat{\Gamma}_\A$) too. 
The set of all the inner unramified components of $\hat{\A}$ is denoted by $\U_\xi\subset \hat{\A}$.
The dual graph of $\hat{\A}$ of inner unramified components is denoted by $\hat{\Gamma}_{\U_\xi}$.
\end{de}

\begin{de}
	Let $\A$ be an arrangement. A character $\xi$ is \emph{inner-cyclic} if it is torsion and $b_1(\hat{\Gamma}_{\U_\xi})>0$. 
\end{de}

\begin{propo}
	An \emph{inner-cyclic arrangement} is the data of a triple $(\A,\xi,\gamma)$, where $\xi$ is an inner-cyclic character on $E_\A$ and $\gamma$ a cycle of $\hat{\Gamma}_{\U_\xi}$.
\end{propo}

\begin{rmk}
	With this point of view, a nearby cycle is a cycle leaving in the boundary of a regular neighbourhood of the union of inner unramified components.
\end{rmk}

%%%%%%%%%%%%%%%%%%%%%%%%%%%%%%%%%%%%%%%%%%%%%
\subsection{Quasi-projective depth}\mbox{}
%%%%%%%%%%%%%%%%%%%%%%%%%%%%%%%%%%%%%%%%%%%%%

The depth of a torsion character is decomposed in two terms, the projective and  quasi-projective depth,
see~\cite{Art_Singularities, Gue_Thesis}. Summarizing, if $\xi$ is a torsion character of order~$n$, then
there is an $n$-fold unbranched ramified quasi-projective cover $\rho:E_\A^\xi\to E_\A$ where
$\sigma: X^\xi\to X^\xi$ generates the deck group of the cover. 
There is a natural isomorphism from $H^1(E_\A;\L_\xi)$ to the eigenspace of $H^1(X^\xi;\CC)$
for $\sigma$ with eigenvalue $e^{\frac{2 i\pi}{n}}$.

Let $\bar{\rho}:X^\xi\to\hat{\PP}^2$ be a smooth model of the projectivization of $\rho$,
where $\bar{\sigma}$ generates the deck group. The inclusion $j_N:E_\A^\xi\hookrightarrow X^\xi$
induces an injection
\begin{equation*}
	j_N^*:\HH^1(X^\xi;\CC) \rightarrow \HH^1(E_\A^\xi;\CC).
\end{equation*}
We denote by $j_{N,\xi}^*$ the restriction of $j_N^*$ 
to the eigenspaces for $\sigma,\bar{\sigma}$ with eigenvalue $e^{\frac{2 i\pi}{n}}$.

\begin{de}
	Let $\xi$ be a torsion character on $\p(E_\A)$.
The \emph{projective depth} of $\xi$ is $\dim\im j_{N,\xi}^*$ while
the \emph{quasi-projective depth} of $\xi$ is:
	\begin{equation*}
		\depth(\xi)=\dim \coker (j_{N,\xi}^*).
	\end{equation*}
\end{de} 

Note that there are known formulas for the computation of the projective depth given by Libgober~\cite{Lib_Characteristic}, 
see~\cite{Art_Combinatorics} for details. Moreover, there is a \emph{finite-time} algorithm to compute
this projective depth for any character. We explain now the method of~\cite{Art_Combinatorics}
to compute the quasi-projective depth. 

We are going to construct a matrix for a \emph{twisted} hermitian intersection $\cdot$ formulas in
a vector space having as base the elements of $\U_\xi$; for an arbitrary
order of this  basis we will consider a square matrix $A_\xi$ of size $\# \U_\xi$,
where coefficients are indexed by elements of $\U_\xi$.
This matrix will depend on the choice of a maximal tree $\T_{\U_\xi}$
of $\hat{\Gamma}_{\U_\xi}$ (maybe a maximal forest, since $\hat{\Gamma}_{\U_\xi}$
is non necessarily connected). For each (oriented) edge $e$ not in $\T_{\U_\xi}$
we consider a cycle $\gamma_e$ consisting on $e$ and a linear
chain of $\hat{\Gamma}_{\U_\xi}$ connecting the final end of $e$ with its starting point.
Let us denote:
\begin{equation*}
\chi(e):=
\begin{cases}
\mathcal{I}(\A,\xi,\gamma)&\text{ if }e\notin\T_{\U_\xi}\\
1&\text{ otherwise.}
\end{cases}
\end{equation*}
The coefficient associated to two components $E$ and $F$ is:
\begin{equation*}
\begin{cases}
E\cdot F&\text{ if }E=F\\
\sum_{e}\chi(e)&\text{ if } E\neq F.
\end{cases}
\end{equation*}
where the sum runs along all the oriented edges from $E$ to $F$.

\begin{thm}[\cite{Art_Singularities}]
	Let $\A$ be an arrangement, and let $\xi$ be a torsion character on $\p(E_\A)$ then:
	\begin{equation*}
		\depth(\xi)=\corank(A_\xi).
	\end{equation*}
\end{thm}

The description of the inclusion map done in~\cite{FloGueMar} (see also Section~\ref{Section_Computability}) and the result on the invariant obtain in Section~\ref{Section_Invariant} allow to compute explicitly the quasi-projective depth of any torsion character.

\begin{ex}
	In the case of the arrangement $\mathscr{C}_7$ with the character $\xi$ defined in 
Example~\ref{Example_Invariant}, the matrix $A_\xi$ is:
	\begin{equation*}
		\left(
			\begin{array}{ccc}
				-1 & 1 & 1 \\
				1 & -1 & \chi_{(\mathscr{C}_7,\xi)}(\gamma)\\
				1 & \chi_{(\mathscr{C}_7,\xi)}(\gamma)\inv & -1
			\end{array}
		\right).
	\end{equation*}
	Example~\ref{Example_Invariant} implies that $\depth(\xi)=2$. Remark that this result is in harmony with the one obtained by D.~Cohen and A.~Suciu in~\cite{CohSuc_Characteristic}.
\end{ex}

%%%%%%%%%%%%%%%%%%%%%%%%%%%%%%%%%%%%%%%%%%%%%
%%%%%%%%%%%%%%%%%%%%%%%%%%%%%%%%%%%%%%%%%%%%%
\section{Computation of the invariant}\label{Section_Computability}
%%%%%%%%%%%%%%%%%%%%%%%%%%%%%%%%%%%%%%%%%%%%%
%%%%%%%%%%%%%%%%%%%%%%%%%%%%%%%%%%%%%%%%%%%%%

The definition of $\mathcal{I}(\A,\xi, \gamma)$ is quite clear, and most
probably it may have a more algebraic description. Nevertheless, its actual definition
is topological, and we need concrete models of the topology of $\A$ and more specifically
of the embedding $B_\A\hookrightarrow E_\A$ at a homology level, see~\cite{FloGueMar,Gue_Thesis}.

In this section we compute the invariant $\mathcal{I}(\A,\xi, \gamma) $ 
from a \emph{wiring diagram} of $\A$. 
As we can see below,
the wiring diagram is a $(3,1)$-dimensional model of the pair $(\PP^2,\A)$ which
contains all the multiple points of $\A$.
Given a cycle $\gamma$ in $\Gamma_\A$, we will use the diagram
 to construct a specific nearby cycle $\tilde{\gamma}$ and calculate its value by the caracter $\xi$. 
This gives a general method to compute the invariant from the equations of $\A$. 

Let us fix an arbitrary line $\ell_0\in\A$ which will
be considered as the line at infinity.

\begin{figure}[ht!]
	\centering
		\begin{tikzpicture}
	\begin{scope}[scale=0.75]
		\begin{scope}

			\ncross{0}{0};\ncross{1}{0};\ncross{2}{0};\ncross{3}{0};
			\ocross{0}{1};\ncross{2}{1};\ncross{3}{1};
			\ncross{0}{2};\mcross{1}{3}{2};
			\ucross{0}{3};\ncross{2}{3};\ncross{3}{3};
			\ncross{0}{4};\ucross{1}{4};\ncross{3}{4};
			\ncross{0}{5};\ocross{1}{5};\ncross{3}{5};
			\ocross{0}{6};\ncross{2}{6};\ncross{3}{6};
			\rcross{0}{7};\ncross{2}{7};\ncross{3}{7};
			\ncross{0}{8};\ucross{1}{8};\ncross{3}{8};
			\ncross{0}{9};\ncross{1}{9};\ucross{2}{9};
			\ncross{0}{10};\ncross{1}{10};\rcross{2}{10};
			\ncross{0}{11};\ocross{1}{11};\ncross{3}{11};
			\ncross{0}{12};\rcross{1}{12};\ncross{3}{12};
			\begin{scope}[shift={(0,0.5)}]
				\draw[fill=black!05] (-0.5,4) -- (-0.5,-0.5) -- (-2,-3.25) -- (-2,1.5) -- (-0.5,4);
				\node at (-0.85,2.5) {$\ell_0$};
			\end{scope}
			
			\draw[dashed] (2.5,-0.5) -- (2.5,3.5);
			\draw[dashed] (7.5,-0.5) -- (7.5,3.5);
			\draw[dashed] (10.5,-0.5) -- (10.5,3.5);
			\draw[dashed] (12.5,-0.5) -- (12.5,3.5);
			
			\draw (0.5,3.5) to[out=90,in=-90] (1.25,3.75) to[out=-90,in=90] (2,3.5);
			\node at (1.25,4.1) {$\beta_0$};
			\draw (3,3.5) to[out=90,in=-90] (5,3.75) to[out=-90,in=90] (7,3.5);
			\node at (5,4.1) {$\beta_1$};
			\draw (8,3.5) to[out=90,in=-90] (9,3.75) to[out=-90,in=90] (10,3.5);
			\node at (9,4.1) {$\beta_2$};
			\draw (11,3.5) to[out=90,in=-90] (11.5,3.75) to[out=-90,in=90] (12,3.5);
			\node at (11.5,4.1) {$\beta_3$};
			
			\node at (15,1.5) {$\CC^2$};
			\node at (15,-2) {$\CC$}; 
			\draw[->] (15,0.5) -- (15,-1);
		\end{scope}
		\begin{scope}[shift={(0,-0.25)}]
			\draw[smooth] (0,-2) node {$\bullet$} node[left] {$x_0$} to[out=0,in=180] (2.5,-2.5) node {$\bullet$} node[below] {$x_1$} to[out=0,in=180] (7.5,-1) node {$\bullet$} node[below] {$x_2$}  to[out=0,in=180] (10.5,-2.5) node {$\bullet$} node[below] {$x_3$} to[out=0,in=180] (12.5,-2) node {$\bullet$} node[below] {$x_4$}  to[out=0,in=180] (13,-2) node[right] {$\nu$};
			
			\draw (0,-0.75) -- (-0.5,-0.75) -- (-2,-3.5) -- (13,-3.5) -- (14.5,-0.75) -- (14,-0.75);
			\draw[dashed] (1,-0.75) -- (0,-0.75) ;
			\draw[dashed] (14,-0.75) -- (13,-0.75) ;
		\end{scope}
	\end{scope}
\end{tikzpicture}
		\caption{Decomposition of a wiring diagram.}
\label{Wiring_Decomposition}
\end{figure}
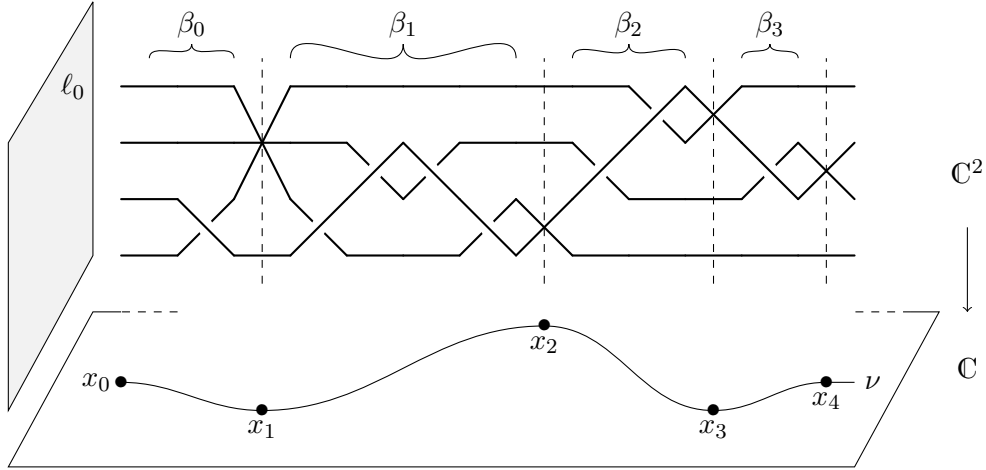

Consider the affine arrangement $\A_0:=\A \setminus\ell_{0}$ of $\CC^2\equiv\PP^2\setminus\ell_0$;
for a projective line $\ell\in\A$, $\ell\neq\ell_0$ we will denote by $L$
the corresponding affine line $\ell\setminus\ell_0$. 
Let $\pi : \CC^2\rightarrow \CC$ be a linear projection, 
\emph{generic} with respect to $\A$ in the sense that for all $i \in \{1,\dots,n \}$, 
the restriction of $\pi_{|L_i}$ is a homeomorphism. 
We choose the coordinates in $\CC^2$ such that $\pi$ is the first projection.
Suppose that the multiple points of $\A$ lie in  different fibers of $\pi$ and that their images $x_i$
satisify  
$\Re(x_1)<\cdots< \Re(x_k)$.

Consider a smooth path $\nu:[0,1]\rightarrow \CC$ whose image starts from a regular value 
$x_0\in\pi(\Tub(L_0))$ and passes through $x_1,\dots,x_k$ in order with $\nu(t_j)=x_j$.
	The \emph{braided wiring diagram} $W_\A$ associated to $\nu$ (voir~\cite{Arv}) is defined by:
	\begin{equation*}
		W_{\A}=\left\{(t,y)\in [0,1]\times\CC\mid\ (\nu(t),y)\in\bigcup\A_0 \right\}.
	\end{equation*}

The space
$W_{\A}\subset[0,1]\times\CC$ is a singular braid with $n$ strings labelled according to the lines, 
whose singular points correspond to the multiple points of $\A$. 
Let us fix a generic projection $\pi_{\RR}:\CC\to\RR$ such that the $n$ strings are generic
outside the singular points. 
For $u=0,\dots, k-1$, the wiring diagram over $(t_u,t_{u+1})$ is identified with a regular
braid $\beta_u$ in the braid group $\Braid_n$ with $n$~strands,  see Figure \ref{Wiring_Decomposition}.

Associated to any singular fiber $\pi^{-1}(t_u)$, containing the multiple point 
$P=  L_{p_1} \cap \dots \cap L_{p_{s(P)}}$, let $\tau_{u}$ be the local positive half-twist  between the strings 
$p_1,\dots, p_{s(P)}$ leaving
 straight the other strings (note that those strings are consecutive).
\begin{figure}[ht!]
	\centering
\begin{tikzpicture}[line cap=round,line join=round]
	\begin{scope}[scale=0.75]

		\begin{scope}[xscale=0.75]
			\draw[thick] (0,0) to[out=0,in=180] (7,3);
			\draw[line width=7pt,color=white] (0,1) to[out=0,in=180] (4.25,3) to[out=0,in=180] (6.25,2) -- (7,2);
			\draw[thick] (0,1) to[out=0,in=180] (4.25,3) to[out=0,in=180] (6.25,2) -- (7,2);
			\draw[line width=7pt,color=white] (0,2)  to[out=0,in=180] (2,3) to[out=0,in=180] (5,1) -- (7,1);
			\draw[thick] (0,2)  to[out=0,in=180] (2,3) to[out=0,in=180] (5,1) -- (7,1);
			\draw[line width=7pt,color=white] (0,3) to[out=0,in=180] (4,0) -- (7,0);
			\draw[thick] (0,3) to[out=0,in=180] (4,0) -- (7,0);
			\node[left] at (0,0) {$s_{p_4}$};
			\node[left] at (0,1) {$s_{p_3}$};
			\node[left] at (0,2) {$s_{p_2}$};
			\node[left] at (0,3) {$s_{p_1}$};
		\end{scope}
		
		\draw[<->] (-2.5,1.5) -- (-1.5,1.5);
		
		\begin{scope}[shift={(-6.5,0)}]
			\draw[thick] (0,0) -- (3,3);
			\draw[thick] (0,1) -- (3,2);
			\draw[thick] (0,2) -- (3,1);
			\draw[thick] (0,3) -- (3,0);
			\node[left] at (0,0) {$L_{p_4}$};
			\node[left] at (0,1) {$L_{p_3}$};
			\node[left] at (0,2) {$L_{p_2}$};
			\node[left] at (0,3) {$L_{p_1}$};
		\end{scope}
	\end{scope}
\end{tikzpicture}
		\caption{The half-twist associated to a singular point of multiplicity 4.}
\label{Correspondence_Point_Twist}
\end{figure}
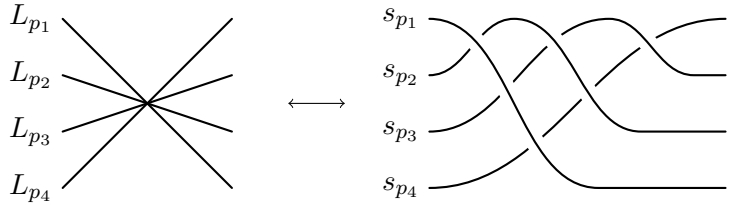

\noindent For all $u,v \in \{0,\dots,k\}$ with $u\neq v$, define
\begin{equation*}
	\beta_{u,v}=\left\{
	\begin{array}{ll}
	\beta_{v-1}\cdot\tau_{v-1}\cdot\beta_{v-2}\cdot\ldots\cdot\beta_{u+1}\cdot\tau_{u+1}\cdot\beta_u &\quad u<v\\
	(\beta_{v,u})\inv & \quad v<u
	\end{array}
	\right..
\end{equation*} 
The braid $\beta_{u,v}$ is obtained from $W_\A$ by taking 
the sub-singular braid bounded by the singular points $P_u$ and $P_v$ 
and replacing the singular crossings by the corresponding local half-twist $\tau$.
It should be noted that the braided wiring diagram encodes in 
fact the braid monodromy of $\A$ relative to the projection $\pi$. 
The operation of replacing singular crossings by half twists corresponds to a
particular choice of a geometric basis of $\pi_1(\CC \setminus \{x_1,\dots,x_k\})$, 
obtained from perturbations of the path $\nu$.

For any braid $\beta\in\Braid_n$ whose strings are labelled by $s_1,\dots, s_n$ and 
all $k,l =1,\dots,n$ with $k\neq l$, let $a_{k,l}(\beta) $ be defined as follows. 
If $p$ is a crossing of $\beta$ between $s_k$ and $s_l$, let $c_{k,l}(p)= 1$ if the string $s_k$ goes over the string $s_l$ and $0$ otherwise.  Let $w(p)$ be the sign of the crossing. Hence
\begin{equation*}
	a_{k,l}(\beta)= \Big(\sum\limits_p w(p) c_{k,l}(p) \Big) \cdot v_k \in H_1(E_\A),
\end{equation*}	
where the sum is taken over all the crossings of $\beta$ between the strings $s_k$ and $s_l$.
\begin{rmk}
The number $a_{k,l}(\beta)$ is the algebraic sum of crossings in $\beta$ between $s_k$ and $s_l$ where $s_k$ goes over $s_l$. 
\end{rmk}

The cycle $\gamma$ of $\Gamma_\A$ can be described by a cyclically ordered sequence of line-vertices 
$v_{\ell_{i_0}},\cdots,v_{\ell_{i_{r}}}$. Up to re-ordering the lines, suppose that $i_0=0$ 
(this will simplify the computation of the invariant, though it is not necessary). 
Let $j_q$ be the sequence of the indices of the point-vertices of $\gamma$, 
such that $\ell_{i_q}\cap\ell_{i_{q+1}} = P_{j_q}$.

Through the wiring diagram $W_\A$, the cycle $\gamma$ can be see directly in $\A$, in a unique way, 
by identifying its vertices 
to their corresponding wires between two multiple points.
Since $\ell_{i_0}=\ell_0$ is the line at infinity, we obtain
 a cycle in $W_\A$ relative to the fiber over $x_0$, uniquely defined. 
Indeed, one starts from a point in $\pi^{-1}(x_0)$, 
along the string labelled by $s_{i_1}$, we go to the singular point containing both $L_{i_1}$ and $L_{i_2}$; 
then along the string $s_{i_2}$, we go to the singular point containing both $L_{i_2}$ and $L_{i_3}$
and so on. We finish by going back to $\pi^{-1}(x_0)$ along the string $s_{i_r}$.

This process \emph{almost} produce an embedding $r(\gamma)$ as in Definition~\ref{def:embedding}.
Such an embedding can be obtained by avoiding the singular points in the wires
corresponding to points~$p$ not in the cycle. In order to construct
the nearby cycle~$\tilde{\gamma}$ we follow the procedure in \cite[Section 4]{FloGueMar}
where the \emph{pushed} cycle in the boundary manifold $B_\A$ is called  a \emph{framed cycle}.
If the wire of some $L_{i_j}$ pass through a point $p\neq P_{j_q}$, then the part of framed cycle
in $\partial\BB_p$ is in the boundary of a regular neighborhood of the trivial knot
$L_{i_j}\cap\partial\BB_p$ in $\partial\BB_p$. Hence the framed cycle $\tilde{\gamma}$ is a nearby cycle.

It is worth noticing that the way to push $\gamma$ off $\A$ depends on conventions, in particular around the
multiple points of $\A$ different from $P_{j_q}$. Different conventions might give another nearby cycle, but whose
  the image by the caracter $\xi$ depends only on $\gamma$.

The next step is to compute the image of $\tilde{\gamma}$ by $i_*$ induced by the inclusion map $B_\A \hookrightarrow E_\A$ at homology level. We use an abelian version of \cite[Theorem 4.3]{FloGueMar},
 given in Lemma \ref{Lemma_InclusionHomology}.

\begin{lem}\label{Lemma_InclusionHomology}
Let $s,t \in \{1,\cdots,n \}$ with $s \neq t$, and $\mathfrak{e}_{s,t}$ be the cycle of $\Gamma_\A$ defined by $v_{L_0},v_{L_s}$ and $v_{L_t}$.  Suppose that the intersection $L_{s} \cap L_t$ lies in the fiber $\pi^{-1}(x_u)$.
There exists a nearby cycle $\tilde{\mathfrak{e}}_{s,t} $ associated with $\mathfrak{e}_{s,t}$ such that
 the image by $i_*: H_1(B_\A) \rightarrow H_1(E_\A)$ induced by the inclusion is
	\begin{equation*}
		i_*(\tilde{\mathfrak{e}}_{s,t})=\sum\limits_{\substack{j=1,\\ j\neq s}}\limits^n a_{j,s}(\beta_{0,u}) + \sum\limits_{\substack{j=1,\\ j\neq t}}\limits^n a_{j,t}(\beta_{u,0})\in\HH_1(E_\A).
	\end{equation*}
\end{lem}

We postpone the proof of Lemma \ref{Lemma_InclusionHomology} to the end of the section.
As a direct consequence, the image of $\tilde{\gamma}$ by the caracter $\xi$ is given in the next lemma.

\begin{propo}\label{Proposition_Computation}
 Let $\tilde{\gamma}$ be a nearby cycle for an inner cyclic realization $(\A,\xi, \gamma)$.
	With the previous notations, we have:
	\begin{equation*}
		\chi_{(\A,\xi)}(\tilde{\gamma}) =\xi\Big(\sum\limits_{q=1}\limits^{r} \sum\limits_{\substack{j=1,\\ j\neq i_q}}\limits^n a_{j,i_q}(\beta_{t_{q-1},t_{q}})\Big).
	\end{equation*}
\end{propo}

\begin{proof}
 For all $i,j =1,\cdots,n$ with $i \neq j$, consider the cycles $\mathfrak{e}_{i,j}$  given by the sequence $v_{L_0},v_{L_i},v_{L_j}$, as in Lemma \ref{Lemma_InclusionHomology}. They generate $H_1(\Gamma_\A)$
 and $\gamma$ can be decomposed as
$ \gamma =\mathfrak{e}_{i_1,i_2} + \mathfrak{e}_{i_2,i_3} + \cdots + \mathfrak{e}_{i_{r-1},i_r}$.
For any choice of nearby cycles $\tilde{\mathfrak{e}}_{i,j}$ associated with $\mathfrak{e}_{i,j}$,  the sum $\tilde{\gamma}=\tilde{\mathfrak{e}}_{i_1,i_2} +\cdots +\tilde{\mathfrak{e}}_{i_{r-1},i_r}$ is a nearby cycle associated with $\gamma$. Remark that Lemma~\ref{Lemma_WellDefine} implies that the image of $i_*(\tilde{\gamma})$ by $\xi$ does not depend on the choice of $\tilde{\gamma}$.
From Lemma~\ref{Lemma_InclusionHomology}, we have:
\begin{equation*}
	\xi\circ i_*(\tilde{\mathfrak{e}}_{i_q,i_{q+1}})=  \xi\Big(\sum\limits_{\substack{j=1,\\ j\neq i_q}}\limits^n  a_{j,i_q}(\beta_{0,t_{q}})  + \sum\limits_{\substack{j=1,\\ j\neq i_{q+1}}}\limits^n  a_{j,i_{q+1}}(\beta_{t_{q},0})\Big).
\end{equation*}
 Since $\beta_{0,t_q}$ and $\beta_{0,t_{q+1}}$ are equal along $\beta_{0,\min(t_q,t_{q+1})}$, then we have:
\begin{equation*}
	\sum\limits_{\substack{j=1,\\ j\neq i_{q+1}}}\limits^n a_{j,i_{q+1}}(\beta_{t_{q},0})
	+\sum\limits_{\substack{j=1,\\ j\neq i_{q+1}}}\limits^n a_{j,i_{q+1}}(\beta_{0,t_{q+1}}) = 
	\sum\limits_{\substack{j=1,\\ j\neq i_{q+1}}}\limits^n a_{j,i_{q+1}}(\beta_{t_{q},t_{q+1}}) + 
	\sum\limits_{q=0}\limits^m \theta_{i_q,i_{q+1}} +  \theta'_{t_q},
\end{equation*}
where $\theta'_{t_q}$ is a product of meridians associated with lines in the support (and then in $\ker(\xi)$) coming from the change of wires at the point $P_{t_q}$. Finally, we have:
\begin{equation*}
	i_*(\tilde{\gamma})=\sum\limits_{q=1}\limits^r \sum\limits_{\substack{j=1,\\ j\neq i_q}}\limits^n a_{j,i_q}(\beta_{t_q,t_{q+1}}) +
	\Theta,
\end{equation*}
where $\Theta$ is an element of $\ker(\xi)$.
\end{proof}

\begin{proof}[Proof of Lemma{\rm~\ref{Lemma_InclusionHomology}}]
As mentioned above, we \emph{push}
a model of $\mathfrak{e}_{s,t}$ from $\A$ to the boundary manifold $B_\A$. 
The cycle  obtained is a \emph{framed} cycle. 
To compute the image of this particular nearby cycle, 
we express it as a product of meridians and a \emph{geometric} cycle $\mathcal{E}_{s,t}$.
 
By construction, geometric cycles bound $2$-cells with holes in the exterior $E_\A$ and counting intersections allows a direct computation of $i_* : H_1(B_\A) \rightarrow H_1(E_\A)$ in this basis. The previous expression can be obtain, by \cite[Proposition 4.1]{FloGueMar}, using the unknotting map $\delta$ (sending framed cycle on the corresponding geometric cycle).
The contribution of each singular point $P$ correspond to the contribution of the strings overcrossing the string $s$ and the string $t$ in $\tau_P$, then an abelian computation shows that:
\begin{equation*}
\delta(\tilde{\mathfrak{e}}_{s,t}))=\mathcal{E}_{s,t}=
-\sum\limits_{q=1}\limits^{m-1}\Big( \sum\limits_{\substack{j=1,\\ j\neq s}}\limits^n a_{j,s}(\tau_{P_q}) \Big)
+ \tilde{\mathfrak{e}}_{s,t}
+ \sum\limits_{q=1}\limits^{m-1}\Big( \sum\limits_{\substack{j=1,\\ j\neq t}}\limits^n a_{j,t}(\tau_{P_q}) \Big) .
\end{equation*}

The model of $\mathfrak{e}_{s,t}$ in $W_\A$ is composed of two parts: a string labelled $s$ from $\pi^{-1}(x_0)$ to $L_{s} \cap L_t \in\pi^{-1}(x_u)$, and a string labelled $t$ from $\pi^{-1}(x_u)$ to $\pi^{-1}(x_0)$.
Let $S_s$ (resp. $S_t$)  be the set of arcs that go over the string $s$ (resp. $t$) 
between $\pi\inv(x_0)$ and $\pi\inv(x_u)$. For each arc $\varsigma$ in $S_s$ or $S_t$, 
let $\sgn(\varsigma)\in\set{\pm 1}$ be the sign of the base $\set{\varsigma,\mathfrak{e}_{s,t}}$ 
($\varsigma$ is oriented from left to right in $W_\A$)
where $\varsigma$ and $\mathfrak{e}_{s,t}$ are considered as vectors in $\RR^2$. 
Let $v_\varsigma\in H_1(E_\A)$ be the meridian of the line corresponding to $\varsigma$.
Once again, an abelian computation shows that:
	\begin{equation*}
		i_*(\mathcal{E}_{s,t})=\mu_{s,t}=\sum\limits_{\varsigma\in S_s} \sgn(\varsigma) v_\varsigma + \sum\limits_{\varsigma\in S_t} \sgn(\varsigma) v_\varsigma =
	\sum\limits_{q=1}\limits^u \sum\limits_{\substack{j=1,\\ j\neq s}}\limits^n a_{j,s}(\beta_{q-1,q}) 
	+ \sum\limits_{q=1}\limits^u \sum\limits_{\substack{j=1,\\ j\neq t}}\limits^n a_{j,t}(\beta_{q-1,q}\inv).
\end{equation*}
Finally, we obtain
\begin{align*}
	i_*(\tilde{\mathfrak{e}}_{s,t}) 
	& =  
		\sum\limits_{q=1}\limits^{u-1}\Big( \sum\limits_{\substack{j=1,\\ j\neq s}}\limits^n a_{j,s}(\tau_{P_q}) \Big) + 
		\sum\limits_{q=1}\limits^u \Big( \sum\limits_{\substack{j=1,\\ j\neq s}}\limits^n a_{j,s}(\beta_{q-1,q}) \Big) +
		\sum\limits_{q=1}\limits^u \Big( \sum\limits_{\substack{j=1,\\ j\neq t}}\limits^n a_{j,t}(\beta_{q-1,q}\inv) \Big) +
		\sum\limits_{q=1}\limits^{u-1}\Big( \sum\limits_{\substack{j=1,\\ j\neq t}}\limits^n a_{j,t}(\tau_{P_q}\inv) \Big),\\
	& = 
		\sum\limits_{\substack{j=1,\\ j\neq s}}\limits^n   a_{j,s}\big( \beta_{0,1} \cdot \tau_{P_1} \cdots \tau_{P_{u-1}} \cdot \beta_{u-1,u}
		\big)  +
		\sum\limits_{\substack{j=1,\\ j\neq s}}\limits^n   a_{j,t}\big(  \beta_{u-1,u}\inv \cdot \tau_{P_{u-1}}\inv \cdots \tau_{P_1}\inv  \cdot \beta_{i_{0},i_1} \inv
		\big),\\
		& =
			\sum\limits_{\substack{j=1,\\ j\neq s}}\limits^n a_{j,s}(\beta_{0,u}) + \sum\limits_{\substack{j=1,\\ j\neq t}}\limits^n a_{j,t}(\beta_{u,0}).
\qedhere
\end{align*}
\end{proof}

%%%%%%%%%%%%%%%%%%%%%%%%%%%%%%%%%%%%%%%%%%%%%
%%%%%%%%%%%%%%%%%%%%%%%%%%%%%%%%%%%%%%%%%%%%%
\section{Example}\label{Section_Examples}
%%%%%%%%%%%%%%%%%%%%%%%%%%%%%%%%%%%%%%%%%%%%%
%%%%%%%%%%%%%%%%%%%%%%%%%%%%%%%%%%%%%%%%%%%%%

The MacLane arrangements are two conjugated arrangements coming from MacLane's matroid \cite{Mac}. 
It is the arrangement with a minimal number of lines such that the combinatorics admits a 
realization over $\CC$ but not over $\RR$ (see \cite[Example 6.6.2(3)]{matroids}). 
These arrangements are constructed as follows. 
Let us consider the 2-dimensional vector space on the field $\FF_3$ of three elements. 
Such a plane contains 9 points and 12 lines, 4 of them pass through the origin. 
Let $\L$ be $\FF_3^2\setminus\set{(0,0)}$ and $\P$
the set of lines in $\FF_3^2$. This provides a line combinatorics $(\L,\P,\Subset)$, 
where for all $\ell\in\L, P\in\P$, we have $P\Subset \ell \Leftrightarrow (\ell \in P, \text{ in }\FF^2_3)$. 
Figure~\ref{MacLane_Combinatorics} represents the ordered MacLane's combinatorics viewed in 
$\FF_3^2$. As an ordered combinatorics, it admits two ordered complex realizations. 

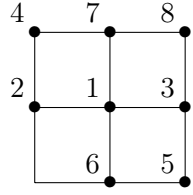
\begin{figure}[ht!]
	\centering
\begin{tikzpicture}
	\begin{scope}
		\node (P2) at (1,0) {$\bullet$};
		\node (P3) at (2,0) {$\bullet$};
		\node (P5) at (0,1) {$\bullet$};
		\node (P0) at (1,1) {$\bullet$};
		\node (P6) at (2,1) {$\bullet$};
		\node (P4) at (0,2) {$\bullet$};
		\node (P1) at (1,2) {$\bullet$};
		\node (P7) at (2,2) {$\bullet$};
		\node[above left] at (1,0) {6};
		\node[above left] at (2,0) {5};
		\node[above left] at (0,1) {2};
		\node[above left] at (1,1) {1};
		\node[above left] at (2,1) {3};
		\node[above left] at (0,2) {4};
		\node[above left] at (1,2) {7};
		\node[above left] at (2,2) {8};
		\draw (0,0) -- (P4.center) -- (P7.center) -- (P3.center) -- (0,0);
		\draw (P2.center) -- (P1.center);
		\draw (P5.center) -- (P6.center);
	\end{scope}
\end{tikzpicture}
		\caption{Ordered MacLane combinatorics, viewed in $\FF_3^2$\label{MacLane_Combinatorics}}
\end{figure}

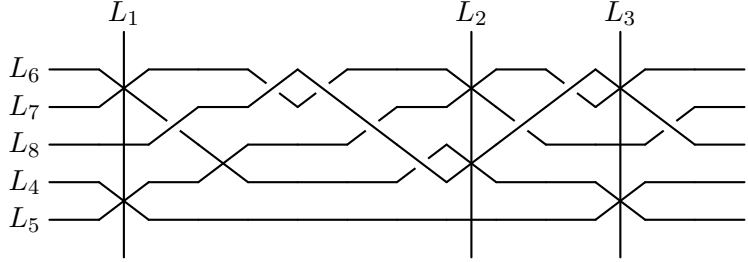
\begin{figure}[ht!]
	\centering
\begin{tikzpicture}	
	\begin{scope}[yscale=0.5,xscale=0.66]
		\ncross{0}{-1};\ncross{1}{-1};\ncross{2}{-1};\ncross{3}{-1};\ncross{4}{-1};
		\rcross{0}{0};\ncross{2}{0};\rcross{3}{0};
		\ncross{0}{1};\ncross{1}{1};\ucross{2}{1};\ncross{4}{1}
		\ncross{0}{2};\rcross{1}{2};\ncross{3}{2};\ncross{4}{2};
		\ncross{0}{3};\ncross{1}{3};\ncross{2}{3};\ucross{3}{3};
		\ncross{0}{4};\ncross{1}{4};\ncross{2}{4};\ocross{3}{4};
		\ncross{0}{5};\ncross{1}{5};\ocross{2}{5};\ncross{4}{5};
		\ncross{0}{6};\ocross{1}{6};\ncross{3}{6};\ncross{4}{6};
		\ncross{0}{7};\rcross{1}{7};\rcross{3}{7};
		\ncross{0}{8};\ncross{1}{8};\ucross{2}{8};\ncross{4}{8};
		\ncross{0}{9};\ncross{1}{9};\ncross{2}{9};\ucross{3}{9};
		\rcross{0}{10};\ncross{2}{10};\rcross{3}{10};
		\ncross{0}{11};\ncross{1}{11};\ocross{2}{11};\ncross{4}{11};
		\ncross{0}{12};\ncross{1}{12};\ncross{2}{12};\ncross{3}{12};\ncross{4}{12};
		\draw[thick,cap=round] (0.5,5) -- (0.5,-1);
		\draw[thick,cap=round] (7.5,5) -- (7.5,-1);
		\draw[thick,cap=round] (10.5,5) -- (10.5,-1);
		\node at (-1.5,0) {$L_5$};	
		\node at (-1.5,1) {$L_4$};	
		\node at (-1.5,2) {$L_8$};	
		\node at (-1.5,3) {$L_7$};	
		\node at (-1.5,4) {$L_6$};	
		\node at (7.5,5.5) {$L_2$};	
		\node at (10.5,5.5) {$L_3$};	
		\node at (0.5,5.5) {$L_1$};
	\end{scope}
\end{tikzpicture}
		\caption{Wiring diagram of extended MacLane arrangement\label{Wiring_Extended_MacLane_Positive}}
\end{figure}
We can give equations to the realizations:
\begin{gather*}
L_1:y-\zeta z=0,\ L_2:y-z=0,\ L_3:y-\bar{\zeta} z=0,\
L_4:x-z=0,\\ L_5:x-\bar{\zeta} y=0,\
L_6:x-\zeta y=0,\ L_7:x-\bar{\zeta} z=0,\ L_8:x-\zeta z=0,
\end{gather*}
where $\zeta$ is a primitive cubic root of unity (its choice determine the realization).
Add to MacLane arrangements a line $L_0$ passing through the intersection points: 
$L_1\cap L_2\cap L_3$ and $L_4\cap L_7\cap L_8$, i.e., $L_0:z=0$ in the above equations. 
We obtain two ordered realization denoted by $\M^+$ and $\M^-$ 
and called respectively positive and negative extended MacLane arrangement 
(the wiring diagram of $\M^+$ is pictured in Figure~\ref{Wiring_Extended_MacLane_Positive}). 
It is not hard to see that the only inner-cyclic characters are $\xi$ and $\xi^{-1}$
where $\xi$ is defined by:
\begin{equation*}
	(v_0,v_1,\dots,v_8)\longmapsto (1,\zeta,\zeta,\zeta,\zeta^2,1,1,\zeta^2,\zeta^2),
\end{equation*}
with corresponding cycle $\gamma:v_{L_0}\rightarrow v_{P_{0,6}}\rightarrow v_{L_6}\rightarrow v_{P_{5,6}}\rightarrow v_{L_5}\rightarrow v_{P_{0,5}}\rightarrow v_{L_0}$ in $\Gamma_{\M^+}$. With the notation of Section~\ref{Section_Computability}, the cycle $\tilde{\mathfrak{e}}_{6,5}^\pm$ is a nearby cycle associated with $\gamma$. In the positive case, the braid $\beta^+\in\Braid_8$ associated with $\tilde{\mathfrak{e}}_{6,5}^+$  is $(\sigma_5\sigma_4\sigma_5)\sigma_3(\sigma_2\sigma_1\sigma_2)\sigma_4\inv$, while in the negative $\beta^-=(\sigma_5\sigma_4\sigma_5)\sigma_3(\sigma_2\sigma_1\sigma_2)\sigma_4$. Then using Proposition~\ref{Proposition_Computation}, the images in the complement of the nearby cycle in both cases are:
\begin{equation*}
	\begin{array}{c}
		i_*(\tilde{\mathfrak{e}}_{6,5}^+)=\sum\limits_{\substack{j=1,\\ j\neq 6}}\limits^8 a_{j,6}(\beta^+) + \sum\limits_{\substack{j=1,\\ j\neq 5}}\limits^8 a_{j,5}(\left(\beta^+\right)\inv)=(v_1-v_8)+(-v_4-v_1)=-v_8-v_4, 
		\\
		i_*(\tilde{\mathfrak{e}}_{6,5}^-)=\sum\limits_{\substack{j=1,\\ j\neq 6}}\limits^8 a_{j,6}(\beta^-) + \sum\limits_{\substack{j=1,\\ j\neq 5}}\limits^8 a_{j,5}(\left(\beta^-\right)\inv)=v_1+(-v_4-v_1)=-v_4.
	\end{array}
\end{equation*} 

Finally, we obtain that :
\begin{equation*}
	\begin{array}{ccc}
		\mathcal{I}(\M^+,\xi,\gamma)=\zeta^2, 
		& \quad & 
		\mathcal{I}(\M^-,\xi,\gamma)(\gamma)=\zeta.
	\end{array}
\end{equation*} 

As a consequence of this paper there is no homeomorphism $\phi:(\PP^2,\M^+)\to(\PP^2,\M^-)$
preserving orders and orientations of the lines. Note that it is a consequence that
this result is already true for the MacLane arrangements;
the fact that MacLane arrangements satisify this property is done using the techniques
in~\cite{ArtCarCogMar_Invariants,Ryb,ry:11}) and the invariant in this paper does not provide
an obstruction.

% \bibliographystyle{amsplain}
% \bibliography{biblio}

%%%%%%%%%%%%%%%%%%%%%%%%%%Bibliography
\providecommand{\bysame}{\leavevmode\hbox to3em{\hrulefill}\thinspace}
\providecommand{\MR}{\relax\ifhmode\unskip\space\fi MR }
\providecommand{\MRhref}[2]{%
  \href{http://www.ams.org/mathscinet-getitem?mr=#1}{#2}
}
\providecommand{\href}[2]{#2}

\end{document}